\newif\ifconfver
\confverfalse     
\confvertrue        

\newif\ifplainver  
\plainvertrue

\ifplainver
    \confverfalse   
\fi

\ifconfver

\documentclass[conference]{IEEEtran}

\IEEEoverridecommandlockouts

\else
    \ifplainver
        \documentclass[11pt]{article}
        \usepackage{fullpage}
    \else
        \documentclass[11pt,draftcls,onecolumn]{IEEEtran}
    \fi
\fi

\usepackage{fullpage,graphicx,algorithm,algorithmic,bm,amsmath,amsthm,amssymb,color,shortcuts_OPT,hyperref,cite,enumitem,microtype,url,cite}

\newcounter{algsubstate}


\newtheorem{Lemma}{Lemma}

\newtheorem{Theorem}{Theorem}

\newtheorem{Assumption}{H\!\!}

\hypersetup{
  colorlinks   = true, 
  urlcolor     = blue, 
  linkcolor    = blue, 
  citecolor    = blue   
}

\makeatletter
\newcommand{\pushright}[1]{\ifmeasuring@#1\else\omit\hfill$\displaystyle#1$\fi\ignorespaces}
\newcommand{\pushleft}[1]{\ifmeasuring@#1\else\omit$\displaystyle#1$\hfill\fi\ignorespaces}
\makeatother

\usepackage{pgfplots}
\usetikzlibrary{arrows,shapes,calc,tikzmark,backgrounds,matrix,decorations.markings}

\pgfplotsset{compat=1.3}

\usepackage{relsize}
\tikzset{fontscale/.style = {font=\relsize{#1}}
    }

\definecolor{lavander}{cmyk}{0,0.48,0,0}
\definecolor{violet}{cmyk}{0.79,0.88,0,0}
\definecolor{burntorange}{cmyk}{0,0.52,1,0}
\definecolor{asuorange}{rgb}{1,0.699,0.0625}
\definecolor{asured}{rgb}{0.598,0,0.199}
\definecolor{asuborder}{rgb}{0.953,0.484,0}
\definecolor{asugrey}{rgb}{0.309,0.332,0.340}
\definecolor{asublue}{rgb}{0,0.555,0.836}
\definecolor{asugold}{rgb}{1,0.777,0.008}

\tikzstyle{server}=[draw, regular polygon, regular polygon sides=4, black!80, fill=black!40,
                       text=violet,minimum width=10pt]
\tikzstyle{worker}=[draw,circle, asublue!80!white, fill = asublue!50!white,
                       text=violet,minimum width=10pt]

    \makeatletter
    \def\multilimits@{\bgroup
  \Let@
  \restore@math@cr
  \default@tag
 \baselineskip\fontdimen10 \scriptfont\tw@
 \advance\baselineskip\fontdimen12 \scriptfont\tw@
 \lineskip\thr@@\fontdimen8 \scriptfont\thr@@
 \lineskiplimit\lineskip
 \vbox\bgroup\ialign\bgroup\hfil$\m@th\scriptstyle{##}$\hfil\crcr}
    \def\Sb{_\multilimits@}
    \def\endSb{\crcr\egroup\egroup\egroup}
\makeatother

\newtheoremstyle{t}         
    {\baselineskip}{2\topsep}      
    {\rm}                   
    {0pt}{\bfseries}  
    {}                      
    { }                      
    {\thmname{#1}\thmnumber{#2}.}

\theoremstyle{t}

\usepackage{environ}

\makeatletter
\DeclareRobustCommand*\cal{\@fontswitch\relax\mathcal}
\makeatother

\begin{document}
\title{On the Convergence of Consensus Algorithms with Markovian Noise and Gradient Bias}
\author{Hoi-To Wai$^\dagger$\thanks{The author is with the Department of SEEM, The Chinese University of Hong Kong, Shatin, Hong Kong. This work is supported by CUHK Direct Grant \#4055113. E-mail: \url{htwai@se.cuhk.edu.hk}}}
\date{\today}

\maketitle

\begin{abstract}
This paper presents a finite time convergence analysis for a decentralized stochastic approximation (SA) scheme. The scheme generalizes several algorithms for decentralized machine learning and multi-agent reinforcement learning. Our proof technique involves separating the iterates into their respective consensual parts and consensus error. The consensus error is bounded in terms of the stationarity of the consensual part, while the updates of the consensual part can be analyzed as a perturbed SA scheme. Under the Markovian noise and time varying communication graph assumptions, the decentralized SA scheme has an expected convergence rate of ${\cal O}(\log T/ \sqrt{T} )$, where $T$ is the iteration number, in terms of squared norms of gradient for nonlinear SA with smooth but non-convex cost function. This rate is comparable to the best known performances of  SA in a centralized setting with a non-convex potential function.\vspace{-.2cm}
\end{abstract} 

\section{Introduction}\vspace{-.3cm}
Decentralized algorithm has become a core tool for control, optimization and machine learning in the increasingly connected world. Among others, a common setting is  multi-agent optimization, where a group of agents on a connected graph seek to minimize a sum of local functions using a common parameter. For convex optimization problems, consensus-based algorithms have been developed in \cite{nedic2009distributed} for deterministic optimization, and the recent results have extended the latter to stochastic/non-convex optimization \cite{bianchi2013performance,di2016next,pu2018distributed,xin2019distributed}.\vspace{-.1cm}

With new machine learning inspired-models such as the training of neural networks and  reinforcement learning, researchers have focused on developing decentralized methods for \emph{non-convex} stochastic optimization \cite{chang2020distributed}. However, methods developed so far for this scenario are limited to using \emph{unbiased} gradient estimates \cite{lian2017can, tang2018d, lu2019gnsd}, requiring i.i.d.~data samples and easy-to-compute gradients. The latter requirements may be an obstacle in deploying the methods in sophisticated problems.\vspace{-.1cm}

The focus of this paper is to relax the unbiased gradient estimate assumption of decentralized stochastic optimization methods. We consider a decentralized stochastic approximation (DSA) scheme whose local update directions are computed from samples of a Markov chain, and the update directions converge asymptotically to a non-gradient mean field, i.e., it is a biased DSA scheme. Our setting is relevant to a decentralized stochastic gradient method with ergodic data as well as multi-agent reinforcement learning.
Our contributions are:\vspace{-.2cm}
\begin{itemize}[leftmargin=3.5mm]
\item For non-convex stochastic optimization, we prove that the DSA scheme with updates from Markov samples converge to a consensual and stationary point at a rate of ${\cal O}( \log T / \sqrt{T} )$, where $T$ is the iteration number. 
\item To analyze the convergence of biased DSA, we develop a decoupling procedure to split the DSA iterates into their consensual part and consensus error, which has a similar favor to the analysis in \cite{bianchi2013performance,mathkar2016nonlinear,vlaski2019distributed}. Such separation allows us to (i) bound the consensus error explicitly; and subsequently (ii) analyze the recursion of the consensual part as a \emph{perturbed} biased SA scheme through adopting prior analysis, e.g., \cite{karimi2019non}. 
\item We show that the biased DSA scheme also converges for time varying topology under a standard bounded connectivity assumption.\vspace{-.1cm}
\end{itemize}
For biased DSA scheme with Markov samples, the closest work related to ours include \cite{bianchi2013performance} which considered a convex optimization related setting; and \cite{sun2019decentralized} which proposed an algorithm that may not achieve exact consensus. 
To the best of our knowledge, this paper provides the first finite-time convergence analysis for DSA scheme with biased updates relying on Markov samples. In addition, this work is related to the recent works on non-asymptotic analysis of SA schemes \cite{kumar2019non,chen2020explicit,doan2020finite}.

\section{Decentralized SA Scheme}\vspace{-.3cm}
Consider a network of $n$ agents which have the goal of obtaining a common (a.k.a.~consensual) solution, $\overline{\prm} \in \RR^d$, which is a stationary point to the optimization problem:\vspace{-.2cm}
\beq \label{eq:opt}  
V^\star = \min_{ \prm \in \RR^d }~V(\prm) \eqdef \frac{1}{n} \sum_{i=1}^n V_i(\prm),\vspace{-.1cm}
\eeq
where $V_i : \RR^d \rightarrow \RR$ is a smooth (but possibly non-convex) function and is lower bounded. Conceptually, the $i$th function $V_i(\prm)$ can be interpreted as the local potential/cost function held by the $i$th agent.
The agents communicate through an undirected, connected simple graph $G = (V,E)$ where $V = \{1,...,n\}$ is the node set, and $E \subseteq V \times V$ is the edge set which includes self-loops. 
The graph is endowed with a symmetric, weighted adjacency matrix ${\bm A}$. We assume
\begin{Assumption} \label{ass:graph}
The matrix ${\bm A} \in \RR_+^{n \times n}$ satisfies:\vspace{-.1cm}
\begin{enumerate}
\item $A_{ij} = 0$ whenever $(i,j) \notin E$. \vspace{-.1cm}
\item ${\bm A} {\bf 1} = {\bm A}^\top {\bf 1} = {\bf 1}$.\vspace{-.1cm}
\item $\| {\bm U}^\top {\bm A} {\bm U} \|_2 \leq 1 - \bar{\rho}$, where $\bar{\rho} \in (0,1]$ and ${\bm U}$ is a projection matrix such that ${\bm I} - \frac{1}{n} {\bf 1}{\bf 1}^\top = {\bm U} {\bm U}^\top$. \vspace{-.2cm}
\end{enumerate}
\end{Assumption}
Note that condition 3) is equivalent to requiring that $\max\{ | \lambda_2( {\bm A}) |, | \lambda_n( {\bm A} ) | \} \leq 1- \bar{\rho}$. Such a weighted adjacency matrix ${\bm A}$ exists if $G$ is connected, e.g., \cite{nedic2009distributed}. 

At iteration $t$, each agent holds a local solution $\prm_i^{(t)}$ and for simplicity, we assume the initialization $\prm_i^{(0)} = \prm_j^{(0)}$ for any $i,j$. Let $X^{t+1} \in {\sf X}$ be a random sample, where ${\sf X}$ is a (discrete or continuous) state space. The decentralized SA (DSA) scheme performs the recursion at all agents $i = 1,...,n$,\vspace{-.1cm}
\beq \label{eq:dsa}
\prm_i^{(t+1)} = \sum_{j=1}^n A_{ij} \prm_j^{(t)} - \gamma_{t+1} {\cal H}_i( \prm_i^{(t)} ; X^{t+1} ),\vspace{-.1cm}
\eeq 
where the first term corresponds to an average consensus step among the neighbors of agent $i$, and ${\cal H}_i( \prm_i^{(t)} ; X^{t+1} )$ is a local stochastic update computed from $\prm_i^{(t)}$, $X^{t+1}$. Furthermore, we denote\vspace{-.1cm}
\beq
{\cal F}_t = \sigma\{ \{ \prm_i^{(0)} \}_{i=1}^n, X^{s},~s=0,1,2,...,t \} \vspace{-.1cm}
\eeq
as the filtration of random elements  up to iteration $t$. Note that $\prm_i^{(t)}$ is measurable w.r.t.~${\cal F}_t$. In the special case of $n=1$, \eqref{eq:dsa} is reduced to the classical SA scheme \cite{robbins1951stochastic}; for general $n>1$, \eqref{eq:dsa} is related to a matrix momentum SA scheme studied in \cite{devraj2018optimal}.

We consider a \emph{biased} DSA scheme in this paper. To describe the setup,
the random samples $\{ X^{t} \}_{t \geq 0}$ forms form  a Markov chain (MC) with the kernel $\PP$ satisfying:\vspace{-.1cm}
\begin{Assumption} \label{ass:ergodic}
The Markov kernel $\PP: {\sf X} \times {\sf X} \rightarrow \RR_+$ generating $\{ X^t \}_{t \geq 0}$ has a unique stationary distribution $\mu : {\sf X} \rightarrow \RR$, and it is irreducible, aperiodic.\vspace{-.2cm}
\end{Assumption}
For any measurable function $f$ on ${\sf X}$, with a slight abuse of notation we define 
$\PP f( X^t ) = \EE[ f( X^{t+1} ) | X^t ] = \int_{\sf X} f(x) \PP( X^t, dx )$. The mean field of  ${\cal H}_i( \prm ; X )$ is 
\beq \textstyle
h_i( \prm ) := \int_{\sf X} {\cal H}_i( \prm ; x ) \, \mu ( dx) .
\eeq
Importantly, the \emph{averaged} mean field 
\beq \textstyle
\overline{h}(\prm) \eqdef \frac{1}{n} \sum_{i=1}^n h_i( \prm )
\eeq 
is related to Problem~\eqref{eq:opt} through:
\begin{Assumption} \label{ass:bias} 
For any $\prm \in \RR^d$, there exists $d_0, c_0 > 0$ such that
\beq \notag
\begin{split}
& \pscal{ \overline{h}(\prm) }{ \grd V( \prm ) } \geq c_0 \| \overline{h}(\prm) \|^2,~ d_0 \| \overline{h}( \prm ) \|^2 \geq \| \grd V( \prm ) \|^2,
\end{split}
\eeq 
where $\pscal{{\bm x}}{{\bm y}} = {\bm x}^\top {\bm y}$ denotes Euclidean inner product.\vspace{-.2cm}
\end{Assumption}
The constants $c_0,d_0$ characterize the multiplicative \emph{bias} of the mean field in view of a stationary solution to \eqref{eq:opt}. It allows the local stochastic update to be \emph{quasi-gradient}, which is relevant when the gradient of $V_i(\prm)$ is hard to obtain. 
Besides, the transient of stochastic update is biased under H\ref{ass:ergodic} with $\EE[ {\cal H}_i( \prm_i^{(t)} ; X^{t+1} ) | {\cal F}_t ] \neq h_i( \prm_i^{(t)})$ for finite $t$. To this end, we assume:
\begin{Assumption} \label{ass:poisson}
For any $i=1,...,n$, $\prm \in \RR^d$, $x \in {\sf X}$, there exists a measurable function $\widehat{\cal H}_i : \RR^d \times {\sf X} \rightarrow \RR^d$ such that
\beq
\widehat{\cal H}_i ( \prm ; x ) - \PP \widehat{\cal H}_i ( \prm ; x ) = {\cal H}_i ( \prm ; x ) - h_i( \prm ).\vspace{-.1cm}
\eeq
\end{Assumption} 
The measurable function in H\ref{ass:poisson} is a solution to the Poisson equation. Such function exists under H\ref{ass:ergodic} and additional conditions on the MC.  For example, we may assume $\PP$ is uniformly ergodic, i.e., for a constant $K$, 
\beq \textstyle \label{eq:uniferg}
\sup_{x \in {\sf X}} \| \PP^t( x, \cdot ) - \mu (\cdot) \|_{\rm TV} \leq K \lambda^{-t},~\forall~t \geq 0,
\eeq
such that $\lambda \in [0,1)$ characterizes the mixing time of $\PP$.
H\ref{ass:poisson} is also satisfied under more relaxed conditions, e.g., geometric ergodicity, see \cite[Ch.~21.2]{douc2018markov}, \cite{glynn1996liapounov}.

We assume that both the local stochastic updates and potential functions are smooth w.r.t.~$\prm$:
\begin{Assumption} \label{ass:lips_h}
For any $i$, the local stochastic update ${\cal H}_i(\prm ; x)$ is $L_h$-Lipschitz w.r.t.~$\prm$, i.e., for any $\prm, \prm' \in \RR^d$,
\beq
\sup_{x \in {\sf X}} \| {\cal H}_i(\prm ; x) - {\cal H}_i(\prm' ; x) \| \leq L_h \| \prm - \prm' \|.
\eeq
Consequently, the mean field map $h_i(\prm)$ is $L_h$-Lipschitz such that $\| h_i( \prm ) - h_i(\prm' ) \| \leq L_h \| \prm - \prm' \|$, $\forall~\prm,\prm' \in \RR^d$. 
\end{Assumption}
\begin{Assumption} \label{ass:lips_V}
The potential function $V(\prm)$ is $L_V$-smooth such that $\| \grd V( \prm ) - \grd V(\prm' ) \| \leq L_V \| \prm - \prm' \|$ for any $\prm,\prm' \in \RR^d$. 
\end{Assumption}
Lastly, we assume the following on ${\cal H}_i( \prm ; x)$:
\begin{Assumption} \label{ass:cons_bd}
For any $\prm = (\prm_1,...,\prm_n)$, there exists $\sigma_o$ such that 
\beq \label{eq:upbd_het}
\begin{split}
& \textstyle \sup_{x \in {\sf X}} \| {\cal H}_i ( \prm_i ; x ) - {\textstyle \frac{1}{n}\sum_{j=1}^n} {\cal H}_j( \prm_j ; x ) \| \\
& \leq \sigma_o \{ {\textstyle \frac{1}{n}} + {\textstyle \frac{1}{n}} \| \overline{h} ( \Tprm_c ) \| + \| \prm_i - \Tprm_c \| \},
\end{split}
\eeq 
for any $i=1,...,n$, where we have $\Tprm_c = {\frac{1}{n}\sum_{j=1}^n} \prm_j$.
\end{Assumption} 
\begin{Assumption} \label{ass:noise_bd}
For any $\Tprm_c \in \RR^d$, there exists $\sigma_h$ such that 
\beq 
\begin{split}
& \textstyle \sup_{x \in {\sf X}} \big\| {\textstyle \frac{1}{n} \sum_{i=1}^n} {\cal H}_i ( \Tprm_c ; x ) - \overline{h} ( \Tprm_c )  \big\| \leq \sigma_h.
\end{split}
\eeq  
\end{Assumption} 
In particular, $\sigma_o$ quantifies the \emph{heterogeneity} of the stochastic updates, while $\sigma_h$ plays a similar role as the variance of $(1/n) \sum_{i=1}^n \{ {\cal H}_i ( \prm ; x ) - h_i( \prm ) \}$. 
Under H\ref{ass:lips_h}, we observe H\ref{ass:cons_bd} can be satisfied if the l.h.s. of \eqref{eq:upbd_het} is upper bounded by ${\cal O}( 1 + \sum_{i=1}^n \| h_i( \prm_i ) \| )$. Our condition H\ref{ass:cons_bd} is considerably weaker than the heterogeneity assumption required by \cite{lian2017can}, as we allow the heterogeneity between the local updates to grow with the norm of mean field.

Notice that 
H\ref{ass:cons_bd}, H\ref{ass:noise_bd} are \emph{uniform bounds} on the norms of error for all $x \in {\sf X}$. They are considerably stronger than those for decentralized stochastic algorithms with i.i.d.~data, e.g., \cite{ghadimi2013stochastic,lian2017can}. However, we remark that this is a caveat for the prior works on SA with Markov noise as well, e.g., \cite{sun2018markov,sun2019decentralized,karimi2019non,srikant2019finite}. We discuss two applications.\vspace{-.1cm}


\subsection{Decentralized SGD with Ergodic Data}\vspace{-.1cm}
 In this case, the $i$th potential function is taken as the following stochastic objective function:\vspace{-.1cm}
\beq \label{eq:sgd_obj}
V_i(\prm) := \EE_{ X_i \sim \mu_i } \big[ V_i( \prm; X_i ) \big].\vspace{-.1cm}
\eeq
The local stochastic update is given by \vspace{-.1cm}
\beq \label{eq:sgd_mkv}
{\cal H}_i( \prm_i^{(t)} ; X_i^{t+1} ) = \grd V_i(  \prm_i^{(t)} ; X_i^{t+1} ), \vspace{-.1cm}
\eeq
such that $\{ X_i^t \}_{t \geq 0}$ is a Markov chain with the kernel $\PP_i: {\sf X}_i \times {\sf X}_i \rightarrow \RR_+$ and a unique stationary distribution $\mu_i : {\sf X}_i \rightarrow \RR_+$. Consequently, the mean field of ${\cal H}_i( \prm_i^{(t)} ; X_i^{t+1} )$ is the gradient $h_i(\prm_i^{(t)}) = \grd V_i(\prm_i^{(t)})$.

Eq.~\eqref{eq:sgd_obj}, \eqref{eq:sgd_mkv} generalize the vanilla decentralized SGD method \cite{lian2017can} to scenarios with non-i.i.d.~(a.k.a.~ergodic) data. As discussed in \cite{duchi2012ergodic,sun2018markov}, the latter is important to applications where data samples are not obtained independently. For example, the data samples are generated using a Markov chain Monte carlo method.  

To see that \eqref{eq:sgd_obj}, \eqref{eq:sgd_mkv} fit the assumptions of this paper, we take ${\sf X} = {\sf X}_1 \times \cdots \times {\sf X}_n$ and form $\PP$ by concatenating the local Markov kernels.
Clearly, H\ref{ass:bias} is satisfied with $c_0=d_0=1$; H\ref{ass:ergodic}, H\ref{ass:poisson} depend on the Markov chain; H\ref{ass:lips_h}, H\ref{ass:lips_V} are related to the smoothness of $V_i( \prm ; x )$ w.r.t.~$\prm$; H\ref{ass:cons_bd} can be satisfied with more homogeneous objective function; H\ref{ass:noise_bd} bounds the noise in estimating the mean field by averaging the local stochastic updates.\vspace{-.1cm} 

%
%
\subsection{Decentralized TD(0) Learning}\vspace{-.1cm} We consider the policy evaluation problem in a multi-agent reinforcement learning (MARL) setting. Our aim is to compute the value function under a policy $\pi$ for an unknown Markov decision process (MDP) using linear function approximation \cite{wai2018multi,sun2019finite,doan2019finite}. 

Consider the MDP at state $x$ with the reward of ${\rm R}(x)$. The agents only observe a local reward ${\rm R}_i(x)$ satisfying ${\rm R}(x) = \frac{1}{n} \sum_{i=1}^n {\rm R}_i(x)$. We aim at approximating the value function as $V(x) = \EE[\sum_{s=0}^\infty \gamma^s {\rm R}(x_s) | x_0 = x ]  \approx \prm^\top \Phi(x)$, where $\prm \in \RR^d$ is the function parameter and $\Phi(x)$ is a `feature' vector.
To find $\prm$, the decentralized TD(0) learning algorithm \cite{sun2019finite,doan2019finite} deploys \eqref{eq:dsa} with the following local stochastic update:
\beq \notag
{\cal H}_i( \prm_i^{(t)} ; x ) = \Phi (x) \big\{ {\rm R}_i(x) + ( \gamma  {\Phi}( x' ) - \Phi(x) )^\top \prm_i^{(t)} \big\}
\eeq
where $x' \in {\sf X}$ denotes the next state drawn from the MDP when the current state is $x$. The terms inside the curly bracket is the temporal difference error. 
We observe that the resultant algorithm is a \emph{linear DSA} scheme. 

To discuss the performance of TD(0), we take $V_i(\prm) = \frac{1}{2}\| \prm - \prm^\star \|^2$, where $\prm^\star$ solves the Bellman equation:
\beq \notag
\EE_\mu[ \Phi(x) ( \gamma  {\Phi}( x' ) - \Phi(x) )^\top ] \prm^\star = \EE_\mu [ \Phi(x) {\rm R}(x) ] .
\eeq
Most of our assumptions can be satisfied by the linear DSA. 
Using \cite[Lemma 3 \& 4]{bhandari2018finite}, H\ref{ass:bias} is satisfied with $c_0 = \frac{1-\gamma}{4}$, and we can show that $d_0 = \EE_\mu[ \| \Phi(x) (\gamma \Phi(x') - \Phi(x) )^\top \|_2 ]^2$. H\ref{ass:ergodic}, H\ref{ass:poisson} are conditions on the Markov chain; H\ref{ass:lips_h}--H\ref{ass:cons_bd} can be satisfied with a bounded $\Phi(x)$, ${\rm R}(x)$.
Lastly, H\ref{ass:noise_bd} can be relaxed in the analysis as the resultant DSA scheme is linear. In the interest of space, we leave the development of the latter case to a future work. 



\section{Finite-Time Analysis of DSA}
For general smooth cost function $V(\prm)$, we consider a random terminating time $\tau(T)$ such that 
\[ \textstyle
{\rm Pr}( \tau(T) = t ) = \gamma_{t+1} / \sum_{s=0}^T \gamma_{s+1},~t=0,...,T,
\]
where $\tau(T) \in \{0,...,T\}$ is selected independently and $T$ is the maximum number of iterations. This is a common stopping criterion proposed in \cite{ghadimi2013stochastic}. For DSA, it can be decided by the agents with a simple consensus protocol before the iterations.
Our main result is summarized as:
\begin{Theorem} \label{thm:mkv}
Under H\ref{ass:graph}--H\ref{ass:noise_bd}, suppose the step size satisfies
\beq \label{eq:step_rule}
\textstyle \sup_{t \geq 0} \gamma_t \leq \min \Big\{ 1, \frac{\bar{\rho} }{2 \sigma_o}, \frac{c_0}{2 \widetilde{C}^{\sf mk}} \Big\}
\eeq
and there exists $\hat{a}$ such that $0 \leq \gamma_t - \gamma_{t+1} \leq \hat{a} \gamma_t^2$. For any $T \geq 0$, it holds that \vspace{-.1cm}
\beq  \label{eq:mf_norm}
\EE[ \| \overline{h} ( \Tprm_c^{(\tau(T))} ) \|^2 ] \leq \frac{ {\rm C}^{\sf tot} }{(c_0/2)\sum_{t=0}^T \gamma_{t+1}},~~
\max_{i=1,...,n} \hspace{-.2cm} \EE[ \| \prm_i^{(\tau(T))} \hspace{-.1cm}  - \Tprm_c^{(\tau(T))} \|] \leq \frac{ \frac{1}{c_0} {\rm C}^{\sf tot} + \frac{3\sigma_o}{2 \bar{\rho}} \sum_{t=0}^{T} \gamma_{t+1}^2 }{ \sum_{t=0}^T \gamma_{t+1}},
\eeq
where we have defined $\Tprm_c^{(t)} := \frac{1}{n} \sum_{i=1}^n \prm_i^{(t)}$, 
\beq \textstyle \notag
{\rm C}^{\sf tot} := V( \frac{\sum_{i=1}^n \Tprm_i^{(0)}}{n} ) - V^\star +  {\rm C}_0^{\sf mk} +  \overline{\rm C}^{\sf mk}  \textstyle \sum_{t=0}^T \gamma_{t+1}^2,
\eeq
and the constants $\widetilde{C}^{\sf mk}$, ${\rm C}_0^{\sf mk}$, $\overline{\rm C}^{\sf mk}$ will be specified in Section~\ref{sec:pf}. The expectation is taken w.r.t.~the terminating iteration $\tau(T)$ and the Markovian randomness.
\end{Theorem}
For the best convergence rate, we may take $\gamma_t = a_0 / \sqrt{t+a_1}$ for some $a_0,a_1>0$. In this case, the theorem shows that the squared norm of mean field  and the consensus error converge at the rate of ${\cal O}( \log T / \sqrt{T} )$. By H\ref{ass:bias} and \eqref{eq:mf_norm}, we have that $\EE[\| \grd V( \Tprm_c^{(\tau(T))} ) \|^2]$ converges at ${\cal O}( \log T / \sqrt{T} )$, i.e., the DSA scheme finds a stationary point to problem \eqref{eq:opt}. Note that this is a standard rate for non-convex stochastic optimization \cite{chang2020distributed}.
Compared to existing works, our convergence rate is similar to a centralized SA scheme, e.g., \cite{karimi2019non}, and it strengthens that of \cite{sun2019decentralized} for DSA with Markov noise, as we provide a convergence rate for \emph{exact consensus}. 

As will be derived later, the constant ${\rm C}^{\sf tot}$ is proportional to ${\cal O}(\bar{\rho}^{-1})$, i.e., related to the spectral gap of the weighted adjacency matrix [cf.~H\ref{ass:graph}] and the magnitude $\sup_{x, \prm}\|\widehat{H}_i(\prm;x)\|$ in  H\ref{ass:poisson}. In the case of uniform MC, the latter is in the order of ${\cal O}(\frac{1}{1-\lambda})$ such that it is related to the mixing time of the Markov chain. Our bound also highlights on the initialization $V( {\sum_{i=1}^n \Tprm_i^{(0)}}/{n} )$.

Instead of analyzing  the convergence of the DSA scheme with a single potential function, in the analysis that follows, we adopt a divide-and-conquer approach similar to \cite{bianchi2013performance,mathkar2016nonlinear,vlaski2019distributed}, where we first decompose the DSA iterate $\prm_1^{(t)}, ..., \prm_n^{(t)}$ into its consensual part and consensus error.
By observing the individual update formulas, we bound the consensus error \emph{separately} as the latter depends on stationarity of the averaged iterate. Subsequently, the consensual part can be analyzed using similar technique as a centralized SA scheme. 
Due to space limitation, we only provide the analysis for general nonlinear DSA under H\ref{ass:graph}--H\ref{ass:noise_bd}. 
\vspace{-.1cm}

\subsection{Proof of Theorem~\ref{thm:mkv}}\label{sec:pf}\vspace{-.1cm}
Define the following $nd$-dimensional vectors\vspace{-.1cm}
\[
\prm^{(t)} := \left( \begin{array}{c}
\prm_1^{(t)} \\ \vdots \\ \prm_n^{(t)}
\end{array} \right),~
{\cal H} (\prm^{(t)}; x) := \left( \begin{array}{c}
{\cal H}_1 (\prm_1^{(t)}; x) \\ \vdots \\ {\cal H}_n (\prm_n^{(t)}; x)
\end{array} \right)\vspace{-.1cm}
\]
as the collection of local solutions and stochastic updates, respectively.  
We rewrite the DSA recursion \eqref{eq:dsa} as:\vspace{-.1cm}
\beq \label{eq:dsa_v}
\prm^{(t+1)} = \big( {\bm A} \otimes {\bm I}_d \big) \prm^{(t)} - \gamma_{t+1} {\cal H} (\prm^{(t)}; X^{t+1}),\vspace{-.1cm}
\eeq
where $\otimes$ denotes the Kronecker product. 

Consider the projection matrix ${\bm I}_n - {\textstyle \frac{1}{n}}{\bm 1} {\bf 1}^\top$ onto the subspace orthogonal to ${\rm span} \{ {\bf 1}_n \}$. As ${\rm rank}( {\bm I}_n - \frac{1}{n} {\bf 1}{\bf 1}^\top) = n-1$, it admits the factorization ${\bm I} - \frac{1}{n} {\bf 1}{\bf 1}^\top = {\bm U} {\bm U}^\top$, where ${\bm U}$ satisfies ${\bm U}^\top {\bm U} = {\bm I}_{n-1}$. 
We let \vspace{-.1cm}
\beq \label{eq:tprm}
\Tprm_c^{(t)} \eqdef ( {\textstyle \frac{1}{n}{\bf 1}^\top } \otimes {\bm I}_d ) \prm^{(t)},~
\Tprm_o^{(t)} \eqdef ( {\bm U}^\top \otimes {\bm I}_d ) \prm^{(t)},\vspace{-.1cm}
\eeq
such that $\Tprm_c^{(t)} \in \RR^d$, $\Tprm_o^{(t)} \in \RR^{(n-1)d}$ denote the consensual component, and the consensus error of $\prm^{(t)}$, respectively. Moreover,
\beq \label{eq:prm_d} 
\prm^{(t)} = ( {\bf 1} \otimes {\bm I}_d ) \Tprm_c^{(t)} + ( {\bm U} \otimes {\bm I}_d ) \Tprm_o^{(t)} .
\eeq
Using \eqref{eq:dsa_v}, \eqref{eq:tprm}, the recursions of the two components in \eqref{eq:prm_d} can be described as\vspace{-.1cm}
\beq \label{eq:decompose}
\begin{split}
\Tprm_c^{(t+1)} & \overset{(a)}{=} \Tprm_c^{(t)} - \gamma_{t+1} ( {\textstyle \frac{1}{n}{\bf 1}^\top } \otimes {\bm I}_d ) {\cal H} (\prm^{(t)}; X^{t+1}) , \\[.1cm]
\Tprm_o^{(t+1)} &  
\overset{(b)}{=} ( {\bm U}^\top {\bm A} {\bm U} \otimes {\bm I}_d ) \Tprm_o^{(t)} \\
& \quad - \gamma_{t+1} ( {\bm U}^\top \otimes {\bm I}_d ) {\cal H} (\prm^{(t)}; X^{t+1}) ,\vspace{-.1cm}
\end{split}
\eeq
%
where (a) used   ${\bf 1}^\top {\bm A} = {\bf 1}^\top$; (b) used  ${\bm U}^\top {\bm A} {\bf 1} = {\bm U}^\top {\bf 1} = {\bm 0}$. 
The recursions in \eqref{eq:decompose} are coupled through the local solutions $\prm^{(t)}$ in the stacked stochastic update. In particular, they allow us to handle the convergence of the respective components as two SA schemes.

Our next step is to derive an intermediate bound on the consensus error $\Tprm_o^{(t+1)}$. A key observation is as follows:
\begin{Lemma} \label{prop:conerr}
Assume H\ref{ass:graph}, H\ref{ass:cons_bd} and the step size satisfies $\gamma_t \leq \frac{\bar{\rho}}{2 \sigma_o}$. If $\prm_i^{(0)} = \prm_j^{(0)}$ for all $i,j$, then it holds for any $t \geq 0$ that\vspace{-.1cm}
\beq \notag 
\begin{split}
& \| \Tprm_o^{(t+1)} \|  \leq \sigma_o \, \sum_{s=0}^{t} \gamma_{s+1} \big(1 - {\textstyle \frac{\bar\rho}{2}} \big)^{t-s} \big\{ 1 + \| \overline{h}( \Tprm_c^{(s)} ) \| \big\}.\vspace{-.1cm}
\end{split}
\eeq
\end{Lemma}
The above lemma shows that the consensus error can be upper bounded by the convolution between an exponential term $(1 - {\textstyle \frac{\bar\rho}{2}})^{t-s}$ and the norm of mean field weighted by the step size as $\gamma_{s+1} \{ 1 + \| \overline{h}( \Tprm_c^{(s)} ) \| \}$.
Importantly, $\| \Tprm_o^{(t)}\|$ decays to zero at the rate of ${\cal O}(\gamma_{t+1})$ provided that $\| \overline{h}( \Tprm_c^{(t)} ) \| \rightarrow 0$. In fact, the above lemma provides a quantitative bound which allows us to decouple consensus error term from the SA error at every iteration.

Next, we focus on the convergence of the consensual component $\Tprm_c^{(t)}$. We observe that
\beq  
\begin{split}
& ( {\textstyle \frac{1}{n}{\bf 1}^\top } \otimes {\bm I}_d ) {\cal H} (\prm^{(t)}; X^{t+1}) \\
& = \overline{h} ( \Tprm_c^{(t)} ) + {\frac{1}{n} \sum_{i=1}^n} {\cal H}_i( \Tprm_c^{(t)}; X^{t+1} ) - \overline{h} (\Tprm_c^{(t)}) + 
{\frac{1}{n} \sum_{i=1}^n} \{ {\cal H}_i( \prm_i^{(t)}; X^{t+1} ) - {\cal H}_i( \Tprm_c^{(t)}; X^{t+1} ) \},
\end{split}
\eeq
where $h( \prm^{(t)} ) := ( h_1(\prm_1^{(t)}); \cdots; h_n(\prm_n^{(t)}) )$. Denote
\[
\begin{split}
{\bm e}_0^{(t)} & \eqdef {\textstyle \frac{1}{n} \sum_{i=1}^n} {\cal H}_i( \Tprm_c^{(t)}; X^{t+1} ) - \overline{h} (\Tprm_c^{(t)}) \\ 
{\bm e}_1^{(t)} & \eqdef {\textstyle \frac{1}{n} \sum_{i=1}^n} \{ {\cal H}_i( \prm_i^{(t)}; X^{t+1} ) - {\cal H}_i( \Tprm_c^{(t)}; X^{t+1} ) \}
\end{split}
\]
Therefore, the recursion of the consensual component $\Tprm_c^{(t)}$ follows that of a perturbed SA scheme:\vspace{-.1cm}
\beq \label{eq:psa}
\Tprm_c^{(t+1)} = \Tprm_c^{(t)}  - \gamma_{t+1} \big\{ \overline{h}( \Tprm_c^{(t)} ) + {\bm e}_0^{(t)} + {\bm e}_1^{(t)} \big\},\vspace{-.1cm}
\eeq
where ${\bm e}_0^{(t)}$ is a perturbation due to the random sample $X^{t+1}$ in estimating the mean field, and ${\bm e}_1^{(t)}$ is bounded by the consensus error.

Our idea is to proceed in a similar fashion as in \cite{karimi2019non}. 
Observe the following lemma:
\begin{Lemma} \label{lem:descent}
Under H\ref{ass:graph}, H\ref{ass:bias}, H\ref{ass:lips_h}, H\ref{ass:lips_V}, H\ref{ass:cons_bd}, H\ref{ass:noise_bd} and assume that $\gamma_t \leq \min\{ \frac{\bar{\rho}}{2 \sigma_o} , 1 \}$. For any $T \geq 0$ and let ${\rm E}_0: = {12 \sigma_o^2 L_h^2} / ( {\bar{\rho} n^2} )$, it holds
\beq \notag
{\begin{split}
& \textstyle \sum_{t=0}^T \gamma_{t+1} \Big( c_0 - \gamma_{t+1} \big\{ {\rm E}_0 + \frac{d_0}{2} + L_V \big\} \Big) \| \overline{h}(\Tprm_c^{(t)}) \|^2 \\[.1cm]
& \textstyle \leq V( \Tprm_c^{(0)} ) - V^\star + \big\{ \sigma_h^2 L_V + {\rm E}_0 \big\} \sum_{t=0}^T \gamma_{t+1}^2 - \sum_{t=0}^T  \gamma_{t+1} \pscal{ \grd V( \Tprm_c^{(t)} ) }{  {\bm e}_0^{(t)}  } .
\end{split}}
\eeq
\end{Lemma}
From the above, we observe that by setting a sufficiently small $\gamma_{t+1}$, it is possible to lower bound the l.h.s.~by $\sum_{t=0}^T \gamma_{t+1} \| \overline{h} (\Tprm_c^{(t)}) \|^2$. Now if the r.h.s.~is finite, the convergence of $\EE[ \| \overline{h} (\Tprm_c^{(\tau(T))}) \|^2 ]$ can be guaranteed.

Our remaining task is to upper bound the inner product $|\EE[ \sum_{t=0}^T  \gamma_{t+1} \pscal{ \grd V( \Tprm_c^{(t)} ) }{  {\bm e}_0^{(t)} }]|$. 
Notice that in the case when $X^{t+1}$ are drawn i.i.d.~from the distribution $\mu$, this inner product is zero. 
Our results below shows that despite that $X^{t+1}$ are not i.i.d., the inner product can still be controlled with an appropriate step size.
\begin{Lemma} \label{lem:markov}
Under H\ref{ass:bias}--H\ref{ass:noise_bd}. Let $|\gamma_t - \gamma_{t+1}| \leq \hat{a} \gamma_t^2$ for some constant $\hat{a}$, and the step sizes satisfies $\gamma_t \leq \min\{ 1, \frac{\bar{\rho}}{2 \sigma_o} \}$. For any $T \geq 0$, it holds 
\beq \notag
\begin{split}
& \textstyle \big| \EE \big[ \sum_{t=0}^T \gamma_{t+1} \pscal{ \grd V( \Tprm_c^{(t)} ) }{ {\bm e}_0^{(t)} } \big] \big|  \leq {\rm C}_0^{\sf mk} + {\rm C}_1^{\sf mk} \sum_{t=0}^T \gamma_{t+1}^2 + {\rm C}_2^{\sf mk} \sum_{t=0}^T \gamma_{t+1}^2 \| \overline{h} (\Tprm_c^{(t)} ) \|^2.
\end{split}
\eeq
Here, ${\rm C}_0^{\sf mk}$, ${\rm C}_1^{\sf mk}$, ${\rm C}_2^{\sf mk}$ are technical and the constants will be defined in \eqref{eq:cmk}.
\end{Lemma}
The above lemma gives a compatible bound of the desired inner product under the scenario of Markovian noise.

Substituting Lemma~\ref{lem:markov} into the conclusion of Lemma~\ref{lem:descent} and rearranging terms yield
\beq \notag
\begin{split}
& \textstyle \sum_{t=0}^T {\textstyle \gamma_{t+1} \Big( c_0 - \gamma_{t+1} \widetilde{\rm C}^{\sf mk} \Big)} \| \overline{h}(\Tprm_c^{(t)}) \|^2 \leq V( \Tprm_c^{(0)} ) - V^\star +  {\rm C}_0^{\sf mk} +  \overline{\rm C}_{\sf mk}  \textstyle \sum_{t=0}^T \gamma_{t+1}^2 =: {\rm C}^{\sf tot}.
\end{split}
\eeq
where $\widetilde{\rm C}^{\sf mk} := {\rm C}_2^{\sf mk} + {\rm E}_0 + \frac{d_0}{2} + L_V$, $\overline{\rm C}^{\sf mk} := {\rm C}_1^{\sf mk} + {\rm E}_0 + \sigma_h^2 L_V$. We also denote the quantity in the r.h.s~by ${\rm C}^{\sf tot}$. If we select the step size according to \eqref{eq:step_rule}, then
\beq \notag 
\begin{split}
\EE[ \| \overline{h}( \Tprm_c^{(\tau(T))} ) \|^2 ] & \textstyle = \sum_{t=0}^T \gamma_{t+1} \| \overline{h}(\Tprm_c^{(t)}) \|^2 / \sum_{t=0}^T \gamma_{t+1} \leq ( (c_0/2) \sum_{t=0}^T \gamma_{t+1} )^{-1} C^{\sf tot}.
\end{split}
\eeq
As such, it concludes our result for the convergence of mean field. Furthermore, note this implies that the norm of gradient of $V(\prm)$ converges [cf.~H\ref{ass:bias}].

Finally, we bound the consensus error. Again, we invoke  Lemma~\ref{prop:conerr} and observe the following
\[
\begin{split}
\textstyle \sum_{t=0}^T \gamma_{t+1} \| \Tprm_o^{(t)} \| & \textstyle \leq \sigma_o \sum_{t=0}^T \gamma_{t+1} \sum_{s=0}^{t-1} \gamma_{s+1} (1 - \frac{\bar{\rho}}{2} )^{t-s} \{ 1 + \| \overline{h}( \Tprm_c^{(s)} ) \| \} \\
& \textstyle \overset{(a)}{\leq} \sigma_o \sum_{s=0}^{T-1} \gamma_{s+1}^2 \{ 1 + \| \overline{h}( \Tprm_c^{(s)} ) \| \}  \sum_{t=s+1}^T (1-\frac{\bar{\rho}}{2})^{t-s} \\
& \textstyle \overset{(b)}{\leq} (3\sigma_o/ 2\bar{\rho}) \sum_{s=0}^{T-1} \gamma_{s+1}^2
+ \sum_{s=0}^{T-1} \frac{\gamma_{s+1}}{2} \| \overline{h}( \Tprm_c^{(s)} ) \|^2.
\end{split}
\]
where (a) involved a change of order in summation and $\gamma_{t+1} \leq \gamma_{s+1}$ as the step size is nonincreasing; (b) involved the condition $\gamma_{s+1} \leq \bar{\rho} / (2\sigma_o)$.
Finally, evaluating the expectation shows that
\beq
\EE[ \| \Tprm_o^{(\tau(T))} \| ] \leq \frac{(3\sigma_o/ 2\bar{\rho}) \sum_{s=0}^{T-1} \gamma_{s+1}^2 + \frac{1}{c_0} {\rm C}^{\sf tot}}{\sum_{t=0}^T \gamma_{t+1}}.
\eeq
The above concludes the proof of Theorem~\ref{thm:mkv}.


\paragraph{Extension to Time-varying Graph}
Our analysis can be extended to scenarios when the communication graph is time varying. Let $G^{(t)} = (V, E^{(t)})$ be a simple, undirected graph which is potentially not connected, where $E^{(t)} \subseteq E$, and the graph is associated with a weighted adjacency matrix ${\bm A}^{(t)}$. 
We replace ${\bm A}$ by ${\bm A}^{(t)}$ in the DSA scheme \eqref{eq:dsa} at iteration $t$, and H\ref{ass:graph} is updated with the following assumption
\begin{Assumption} \label{ass:graph_tv}
For any $t \geq 0$, the matrix ${\bm A}^{(t)} \in \RR^{n \times n}$ satisfies:
\begin{enumerate}
\item $A_{ij}^{(t)} = 0$ whenever $(i,j) \notin E^{(t)}$.\vspace{-.1cm}
\item ${\bm A}^{(t)} {\bf 1} = ({\bm A}^{(t)})^\top {\bf 1} = {\bf 1}$.\vspace{-.1cm}
\item $\exists B \geq 0$ with $\| {\bm U}^\top {\bm A}^{(t+B-1)} \cdots {\bm A}^{(t)} {\bm U} \|_2 \leq 1 - \bar{\rho}$, where $\bar{\rho} \in (0,1]$. 
\end{enumerate}\vspace{-.2cm}
\end{Assumption}
The last condition can be guaranteed under the `bounded communication' setting \cite{nedic2009distributed}, i.e., when the combined graph $(V, E^{(t)} \cup \cdots E^{(t+B)})$ is connected for any $t \geq 0$. 

As ${\bm A}^{(t)}$ remains doubly stochastic, the decomposition in \eqref{eq:decompose} is still valid. We can then extend Lemma~\ref{prop:conerr} to bound the consensus error using a blocking argument; see the discussion in Appendix~\ref{app:conerr}. The proof for Theorem~\ref{thm:mkv} can be modified accordingly and we obtain the same convergence rate for the time varying graph setting.\vspace{-.1cm}

\section{Conclusions}\vspace{-.1cm}
In this paper, we have studied the convergence of a biased decentralized stochastic approximation (DSA) scheme. The scheme is a multi-agent optimization algorithm relying on biased, stochastic updates that approximate the gradient of a smooth cost function. Here, the biasednesses stem from taking Markov samples and quasi-gradients in the updates. We prove that DSA finds a consensual and stationary point to the cost function at a rate of ${\cal O}(\log T/ \sqrt{T})$, where $T$ is the maximum iteration number. Future works include extending to asynchronous, gradient tracking DSA, state-controlled Markov chain, etc..

\ifplainver
\appendix
\else
\appendices
\fi

\section{Proof of Lemma~\ref{prop:conerr} \& Its Extension}\label{app:conerr}
From the recursion \eqref{eq:decompose}, we observe that
\beq \label{eq:recur_pf_o}
\begin{split}
\| \Tprm_o^{(t+1)} \| & \leq \| ( {\bm U}^\top {\bm A}  {\bm U} \otimes {\bm I} ) \Tprm_o^{(t)}  \|  + \gamma_{t+1} \|  ({\bm U}^\top \otimes {\bm I}) {\cal H}( \prm^{(t)} ; X^{t+1} ) \| .
\end{split}
\eeq
Using H\ref{ass:graph}, we observe the contraction
\beq
\begin{split}
& \| ( {\bm U}^\top {\bm A} {\bm U} \otimes {\bm I} ) \Tprm_o^{(t)}  \|  \leq \| ( {\bm U}^\top {\bm A}  {\bm U} \otimes {\bm I} ) \|_2 \| \Tprm_o^{(t)}  \|  \leq (1 - \bar{\rho}) \| \Tprm_o^{(t)}  \|.
\end{split}
\eeq
Using H\ref{ass:cons_bd}, we bound the second term in \eqref{eq:recur_pf_o} as:
\beq \notag
\begin{split}
\|  ({\bm U}^\top \otimes {\bm I}) {\cal H}( \prm^{(t)} ; X^{t+1} ) \| & \textstyle \leq \|  {\cal H}( \prm^{(t)} ; X^{t+1} ) - ( \frac{1}{n}{\bf 1}{\bf 1}^\top \otimes {\bm I}_d )  {\cal H}( \prm^{(t)} ; X^{t+1} ) \| \\[.1cm]
& \leq \textstyle \sigma_o \sum_{i=1}^n \{ {\textstyle \frac{1}{n}} + {\textstyle \frac{1}{n}} \| \overline{h} ( \Tprm_c^{(t)} ) \|] + \| \prm_i^{(t)} - \Tprm_c^{(t)} \|  \},
\end{split}
\eeq
which can be further simplified as $\sigma_o \big\{ 1 + \| \overline{h} ( \Tprm_c^{(t)} ) \| + \| \Tprm_o^{(t)} \| \big\}$.
Substituting into \eqref{eq:recur_pf_o} yields
\beq  \label{eq:recur_finn}
\begin{split}
 \| \Tprm_o^{(t+1)} \| & \leq (1 - \bar{\rho} + \gamma_{t+1} \sigma_o )  \| \Tprm_o^{(t)} \|  + \gamma_{t+1} \sigma_o \{1 +  \| \overline{h} ( \Tprm_c^{(t)} ) \| \}.
\end{split}
\eeq
Setting $\gamma_{t+1} \leq \frac{ \bar{\rho} }{2 \sigma_o } $ yields that $1 - \bar{\rho} + \gamma_{t+1} \sigma_o \leq 1 - \frac{\bar{\rho}}{2}$. Solving the recursion and noticing that $\Tprm_o^{(0)} = {\bm 0}$ yield the desired bound.

\paragraph{Extension to Time-varying Topology} Under the relaxed condition H\ref{ass:graph_tv}, we apply a blocking argument to derive the result as in Lemma~\ref{prop:conerr}. In particular, denote $\Prm(m,n) := \| \Tprm_o^{(m)} \| + \cdots \| \Tprm_o^{(n)} \|$, we can show:
\[
\begin{split}
& \Prm( t+1, t+B ) \leq (1 - \bar{\rho}) \Prm( t-B+1, t ) \\
& \textstyle+ \sigma_o \gamma_{t-B+2} \big\{ \Prm(t-B+1,t) + \cdots + \Prm(t,t+B-1) \big\} + \sigma_o B  \sum_{s=t-B+1}^{t+B-1} \gamma_{s+1} \big\{ 1 + \| \overline{h}( \Tprm_c^{(s)} ) \| \big\},
\end{split}
\]
which implies that
\[
\begin{split}
& \Prm( t+1, t+B ) \leq \frac{ 1 - \bar{\rho} +  \sigma_o B \gamma_{t-B+2} }{ 1 - \sigma_o B \gamma_{t-B+2}  } \Prm( t-B+1, t ) + \frac{\sigma_o B}{ 1- \sigma_o B \gamma_{t-B+2}}  \sum_{s=t-B+1}^{t+B-1} \gamma_{s+1} \big\{ 1 + \| \overline{h}( \Tprm_c^{(s)} ) \| \big\}.
\end{split}
\]
Setting a sufficiently small step size $\gamma_t$ allows us to derive a similar recursion as \eqref{eq:recur_finn} for $\Prm( t+1, t+B )$. Solving it yields a convolution bound as in Lemma~\ref{prop:conerr}.  

\section{Proof of Lemma~\ref{lem:descent}}
Using the $L_V$-smoothness of $V(\cdot)$ [cf.~H\ref{ass:lips_V}], we observe
\beq \notag
\begin{split}
V( \Tprm_c^{(t+1)} ) & \leq V( \Tprm_c^{(t)} ) + \gamma_{t+1}^2 \frac{L_V}{2} \| \overline{h}( \Tprm_c^{(t)} ) + {\bm e}_0^{(t)} + {\bm e}_1^{(t)}\|^2  - \gamma_{t+1} \pscal{ \grd V( \Tprm_c^{(t)} ) }{ \overline{h}( \Tprm_c^{(t)} ) + {\bm e}_0^{(t)} + {\bm e}_1^{(t)} } \\[.1cm]
& \hspace{-1.6cm} \leq V( \Tprm_c^{(t)} ) - \gamma_{t+1} \big( c_0 - \gamma_{t+1} L_V \big) \| \overline{h}(\Tprm_c^{(t)}) \|^2 + \gamma_{t+1}^2 L_V \|  {\bm e}_0^{(t)} + {\bm e}_1^{(t)}\|^2 - \gamma_{t+1} \pscal{ \grd V( \Tprm_c^{(t)} ) }{  {\bm e}_0^{(t)} + {\bm e}_1^{(t)} }
\end{split}
\eeq
Summing up the above from $t=0$ to $t=T$ yields
\beq \notag
\begin{split}
& \textstyle \sum_{t=0}^T \gamma_{t+1} \Big( c_0 - \gamma_{t+1} L_V \Big) \| \overline{h}(\Tprm_c^{(t)}) \|^2 \\
& \textstyle\leq V( \Tprm_c^{(0)} ) - V^\star + 2 L_V \sum_{t=0}^T \gamma_{t+1}^2 \big\{ \| {\bm e}_0^{(t)}\|^2 + \| {\bm e}_1^{(t)} \|^2 \big\} - \sum_{t=0}^T  \gamma_{t+1} \pscal{ \grd V( \Tprm_c^{(t)} ) }{  {\bm e}_0^{(t)} + {\bm e}_1^{(t)} } 
\end{split}
\eeq
By H\ref{ass:lips_h}, we observe
\beq
\begin{split}
& \| {\bm e}_1^{(t)} \| \leq (L_h/n) \, \| \prm^{(t)} - ({\bf 1} \otimes {\bm I}_d) \Tprm_c^{(t)} \| = (L_h/n) \, \| ( {\bm U} \otimes {\bm I}_d ) \Tprm_o^{(t)} \| \leq (L_h/n) \| \Tprm_o^{(t)} \|
\end{split}
\eeq
Also, applying H\ref{ass:noise_bd} and re-arranging terms show that
\beq \notag
\begin{split}
&  \sum_{t=0}^T \gamma_{t+1} \Big( c_0 - \gamma_{t+1} L_V \Big) \| \overline{h}(\Tprm_c^{(t)}) \|^2 \\
& \textstyle \leq V( \Tprm_c^{(0)} ) - V^\star - \sum_{t=0}^T  \gamma_{t+1} \pscal{ \grd V( \Tprm_c^{(t)} ) }{  {\bm e}_0^{(t)} + {\bm e}_1^{(t)} }  + 2 L_V \sum_{t=0}^T \gamma_{t+1}^2 \big\{ \sigma_h^2 + \frac{L_h^2}{n^2} \| \Tprm_o^{(t)} \|^2 \big\} 
\end{split}
\eeq
Moreover, using H\ref{ass:bias} we observe 
\[
\gamma_{t+1} \pscal{ \grd V( \Tprm_c^{(t)} ) }{ {\bm e}_1^{(t)} } \leq \frac{\gamma_{t+1}^2 d_0}{2} \| \overline{h}( \Tprm_c^{(t)} ) \|^2 + \frac{1}{2} \| {\bm e}_1^{(t)} \|^2 
\]
Re-arranging terms again and using $\gamma_t \leq 1$ show
\beq \notag
\begin{split}
& \textstyle \sum_{t=0}^T \gamma_{t+1} \Big( c_0 - \gamma_{t+1} \big\{ \frac{d_0}{2} + L_V \big\} \Big) \| \overline{h}(\Tprm_c^{(t)}) \|^2 \\
& \textstyle \leq V( \Tprm_c^{(0)} ) - V^\star + \sigma_h^2 L_V \sum_{t=0}^T \gamma_{t+1}^2  + \sum_{t=0}^T  \frac{3 L_h^2}{n^2} \| \Tprm_o^{(t)} \|^2 - \sum_{t=0}^T  \gamma_{t+1} \pscal{ \grd V( \Tprm_c^{(t)} ) }{  {\bm e}_0^{(t)}  } 
\end{split}
\eeq
Next, we need to upper bound $\sum_{t=0}^T  \| \Tprm_o^{(t)} \|^2 $ with  Lemma~\ref{prop:conerr} and \ref{lem:sum}, we obtain
\beq \label{eq:Tprmo}
\begin{split}
&  \sum_{t=0}^T \| \Tprm_o^{(t)} \|^2 \leq \frac{4 \sigma_o^2 }{\bar{\rho}} \sum_{t=0}^T \gamma_{t+1}^2 \big\{ 1 + \| \overline{h}( \Tprm_c^{(t)} ) \|^2 \big\}.
\end{split}
\eeq
Define the constant ${\rm E}_0 \eqdef {12 \sigma_o^2 L_h^2} / ( {\bar{\rho} n^2} )$
and substituting into the previous inequality, we obtain
\beq \notag
{\begin{split}
& \textstyle \sum_{t=0}^T \gamma_{t+1} \Big( c_0 - \gamma_{t+1} \big\{ {\rm E}_0 + \frac{d_0}{2} + L_V \big\} \Big) \| \overline{h}(\Tprm_c^{(t)}) \|^2 \\
& \textstyle \leq V( \Tprm_c^{(0)} ) - V^\star + \big\{ \sigma_h^2 L_V + {\rm E}_0 \big\} \sum_{t=0}^T \gamma_{t+1}^2 - \sum_{t=0}^T  \gamma_{t+1} \pscal{ \grd V( \Tprm_c^{(t)} ) }{  {\bm e}_0^{(t)}  } 
\end{split}}
\eeq
This is the desirable bound for the lemma.

\section{Proof of Lemma~\ref{lem:markov}}
We begin the proof by using the solution to Poisson equation defined in H\ref{ass:poisson}. We have
\beq \notag
\begin{split}
 {\bm e}_0^{(t)} & = {\textstyle \frac{1}{n} \sum_{i=1}^n} \big\{ {\cal H}_i( \Tprm_c^{(t)}; X^{t+1} ) - {h}_i (\Tprm_c^{(t)}) \big\}  = {\textstyle \frac{1}{n} \sum_{i=1}^n} \big\{ \hat{\cal H}_i( \Tprm_c^{(t)}; X^{t+1} ) - \PP \hat{H}_i (\Tprm_c^{(t)} ; X^{t+1} ) \big\}.
\end{split}
\eeq
The above allows us to derive the decomposition
\beq \notag
\begin{split}
& \sum_{t=0}^T \gamma_{t+1} \pscal{ \grd V( \Tprm_c^{(t)} ) }{ {\bm e}_0^{(t)} }  \equiv {\frac{1}{n} \sum_{i=1}^n} \{A_1^i + A_2^i + A_3^i + A_4^i \},
\end{split}
\eeq
where
\beq \notag
{\small \begin{split}
A_1^i & \eqdef \sum_{t=0}^T \gamma_{t+1} \pscal{ \grd V( \Tprm_c^{(t)} ) }{ \hat{\cal H}_i( \Tprm_c^{(t)} ; X^{t+1} ) - \PP \hat{\cal H}_i( \Tprm_c^{(t)} ; X^{t} ) }, \\
A_2^i & \eqdef \sum_{t=0}^T \gamma_{t+1} \pscal{ \grd V( \Tprm_c^{(t)}) } { \PP \hat{\cal H}_i( \Tprm_c^{(t)} ; X^{t} ) - \PP \hat{\cal H}_i( \Tprm_c^{(t-1)} ; X^{t} ) }, \\
A_3^i & \eqdef \sum_{t=0}^T \gamma_{t+1} \pscal{ \grd V( \Tprm_c^{(t)}) } { \PP \hat{\cal H}_i( \Tprm_c^{(t-1)} ; X^{t} ) } - \sum_{t=0}^{T-1} \gamma_{t+2} \pscal{ \grd V( \Tprm_c^{(t+1)}) } { \PP \hat{\cal H}_i( \Tprm_c^{(t)} ; X^{t+1} ) }, \\
A_4^i & \eqdef \sum_{t=0}^{T-1}  \pscal{ \gamma_{t+2} \grd V( \Tprm_c^{(t+1)}) -  \gamma_{t+1} \grd V( \Tprm_c^{(t)}) } { \PP \hat{\cal H}_i( \Tprm_c^{(t)} ; X^{t+1} ) }  - \gamma_{T+1} \pscal{ \grd V( \Tprm_c^{(T)} ) }{ \PP \hat{\cal H}_i( \Tprm_c^{(T)} ; X^{T+1} ) }.
\end{split}}
\eeq
We have set $\prm^{(0)} = \prm^{(-1)}$ as a convention in the above.
Next, we upper bound the above terms as follows. 

Firstly, due to the Martingale property with $\EE[ \pscal{ \grd V( \Tprm_c^{(t)} ) }{ \hat{\cal H}_i( \Tprm_c^{(t)} ; X^{t+1} ) - \PP \hat{\cal H}_i( \Tprm_c^{(t)} ; X^{t} ) } | {\cal F}_t ] = 0$, we have 
\beq
\EE[ {\textstyle \frac{1}{n} \sum_{i=1}^n} A_1^i ] = 0,~\forall~i.
\eeq
Secondly, note that H\ref{ass:lips_h} implies that $\PP \hat{\cal H}_i(\prm; x)$ is $\bar{L}_h$-Lipschitz w.r.t.~$\prm$, for some constant $\bar{L}_h$ \cite{fort2011convergence}. As such, 
\beq
\begin{split}
A_2^i & \leq \sum_{t=0}^T \gamma_{t+1} \bar{L}_h \| \grd V( \Tprm_c^{(t)} ) \| \| \Tprm_c^{(t)} - \Tprm_c^{(t-1)} \| \leq \sum_{t=0}^T \gamma_{t+1} {d_0}^{\frac{1}{2}} \bar{L}_h \| \overline{h}( \Tprm_c^{(t)} ) \| \| \Tprm_c^{(t)} - \Tprm_c^{(t-1)} \|.
\end{split}
\eeq 
Taking the summation over $i$ and dividing by $n$ yield
\beq \notag
\frac{1}{n} {\sum_{i=1}^n} A_2^i \leq  \sum_{t=0}^T \gamma_{t+1} {d_0}^{\frac{1}{2}} \bar{L}_h \| \overline{h}( \Tprm_c^{(t)} ) \| \| \Tprm_c^{(t)} - \Tprm_c^{(t-1)} \|.
\eeq
Notice that 
\beq \notag
\begin{split}
\Tprm_c^{(t)} - \Tprm_c^{(t-1)} & = - \gamma_t \big\{ \overline{h}( \Tprm_c^{(t-1)} ) + {\bm e}_0^{(t-1)} + {\bm e}_1^{(t-1)} \big\}.
\end{split}
\eeq
We observe that 
\beq 
\| {\bm e}_0^{(t-1)} \| 
\leq \sigma_h,~~ \| {\bm e}_1^{(t-1)} \| \leq (L_h/n) \| \Tprm_o^{(t-1)} \|.
\eeq
As such,
\beq 
\begin{split} 
 \frac{1}{{d_0}^{\frac{1}{2}} \bar{L}_h}  \frac{1}{n} \sum_{i=1}^n A_2^i & \leq  \frac{ L_h}{n} \sum_{t=0}^T \gamma_{t+1} \gamma_t \| \overline{h}( \Tprm_c^{(t)} ) \| \| \Tprm_o^{(t-1)} \| + \sum_{t=0}^T \gamma_{t+1} \gamma_t \| \overline{h}( \Tprm_c^{(t)} ) \| \big\{ \sigma_h + \| \overline{h}( \Tprm_c^{(t-1)} ) \| \} \\[.1cm]
& \leq \big( \frac{ L_h}{2n} + 2 \big) \sum_{t=0}^T \gamma_{t+1}^2\| \overline{h}( \Tprm_c^{(t)} ) \|^2 +\sum_{t=0}^T \| \Tprm_o^{(t-1)} \|^2 +  \sigma_h^2 \sum_{t=0}^T \gamma_{t}^2 \\[.1cm]
& \leq \big( 2 + \frac{L_h}{2n} + \frac{4 \sigma_o^2}{ \bar{\rho} } \big) \sum_{t=0}^T \gamma_{t+1}^2\| \overline{h}( \Tprm_c^{(t)} ) \|^2 + \big( \sigma_h^2 + \frac{4 \sigma_o^2}{ \bar{\rho} } \big) \sum_{t=0}^T \gamma_{t+1}^2.
\end{split}
\eeq
To analyze the last two terms $A_3^i$, $A_4^i$, we denote 
\beq
\bm{\mathcal{E}} = {\textstyle \frac{1}{n}} {\bf 1}^\top \otimes {\bm I}_d,~~\prm_c^{(t)} := ({\bf 1} \otimes {\bm I}_d) \Tprm_c^{(t)},
\eeq 
such that $\PP \hat{\cal H}( \prm_c ; x ) = ( \PP \hat{\cal H}_1( \Tprm_c ; x ); \cdots ; \PP \hat{\cal H}_n( \Tprm_c ; x ) )$.

Thirdly, we observe that from \cite[Lemma 4.2]{fort2011convergence}, under H\ref{ass:ergodic}, H\ref{ass:poisson}, H\ref{ass:noise_bd}, it can be shown for any $\prm_c = ({\bf 1} \otimes {\bm I}_d) \Tprm_c$, with $\Tprm_c \in \RR^d$, and $x \in {\sf X}$ that:
\beq
\| \PP \hat{\cal H}( \prm_c ; x ) \| \leq K_{\PP}.
\eeq
Here, $K_{\PP}$ depends on the mixing time of the Markov chain, e.g., it is proportional to $\frac{1}{1-\lambda}$ under the uniform ergodicity condition \eqref{eq:uniferg}.
Therefore, 
\beq 
\begin{split} \textstyle
\frac{1}{n} \sum_{i=1}^n A_3^i & = \gamma_1 \pscal{ \grd V( \Tprm_c^{(0)} ) }{ \bm{\mathcal{E}}  \PP \hat{\cal H}( \prm_c^{(0)} ; X^0) } \\
& \leq \gamma_1 K_{\PP} \| \grd V( \Tprm_c^{(0)} ) \| . 
\end{split}
\eeq
Fourthly, using   $|\gamma_{t+2} - \gamma_{t+1}| \leq \hat{a} \gamma_{t+1}^2$, we have 
\beq \notag
{\small
\begin{split}
\frac{1}{n} \sum_{i=1}^n A_4^i  & \leq \hat{a} d_0^{\frac{1}{2}} \sum_{t=0}^{T-1} \gamma_{t+1}^2 \| \overline{h}( \Tprm_c^{(t+1)}) \| \| \bm{\mathcal{E}}  \PP \hat{\cal H} ( \prm_c^{(t)}; X^{t+1} ) \| + L_V \sum_{t=0}^{T-1} \gamma_{t+1}  \| \Tprm_c^{(t+1)} - \Tprm_c^{(t)} \| \| \bm{\mathcal{E}}  \PP \hat{\cal H} ( \prm_c^{(t)}; X^{t+1} ) \| \\
& \quad + \gamma_{T+1} \| \grd V( \Tprm_c^{(T)} ) \| \| \bm{\mathcal{E}} \PP \hat{\cal H}_i( \prm_c^{(T)} ; X^{T+1} )  \| \\[.1cm]
& \leq K_{\PP} \sum_{t=0}^{T-1} \big\{ \hat{a} d_0^{\frac{1}{2}} \gamma_{t+1}^2 \| \overline{h}( \Tprm_c^{(t+1)}) \|  + L_V \gamma_{t+1}  \| \Tprm_c^{(t+1)} - \Tprm_c^{(t)} \| \big\}  + \gamma_{T+1} K_{\PP} \| \grd V( \Tprm_c^{(T)} ) \| .
\end{split}}
\eeq
To bound $\frac{1}{n} \sum_{i=1}^n A_4^i$, we observe that
\beq \notag
\begin{split}
& \textstyle \sum_{t=0}^{T-1} \gamma_{t+1}^2 \| \overline{h}( \Tprm_c^{(t+1)}) \|  \leq \sum_{t=0}^{T-1} \gamma_{t+1}^2 \big\{ 1 + \| \overline{h}( \Tprm_c^{(t+1)}) \|^2 \big\},
\end{split}
\eeq
the latter can be further simplified as
\[ \textstyle
\sum_{t=0}^{T-1} \gamma_{t+1}^2 \| \overline{h}( \Tprm_c^{(t+1)}) \|^2 \leq a^2 \sum_{t=0}^T \gamma_{t+1}^2 \| \overline{h}( \Tprm_c^{(t)}) \|^2. \]
Moreover,
\beq \notag
{\small \begin{split}
& \sum_{t=0}^{T-1} \gamma_{t+1}  \| \Tprm_c^{(t+1)} - \Tprm_c^{(t)} \|  = \sum_{t=0}^{T-1} \gamma_{t+1}^2 \| \overline{h}( \Tprm_c^{(t)} ) + {\bm e}_0^{(t)} + {\bm e}_1^{(t)} \| \\
& \leq \sum_{t=0}^{T-1} \gamma_{t+1}^2 \big\{ \sigma_h + \| \overline{h}( \Tprm_c^{(t)} ) \| + {\textstyle \frac{L_h}{n}} \| \Tprm_o^{(t)} \| \} \\
& \leq \sum_{t=0}^{T-1} \gamma_{t+1}^2 \big\{ \sigma_h + \frac{1 + \frac{L_h}{n}}{2} + {\textstyle \frac{1}{2}} \| \overline{h}( \Tprm_c^{(t)} ) \|^2 + {\textstyle \frac{L_h}{2 n}} \| \Tprm_o^{(t)} \|^2 \} \\
& \leq \sum_{t=0}^{T-1} \gamma_{t+1}^2 \big\{ {\textstyle \frac{\bar{\rho} n (1 + 2 \sigma_h) + L_h + 4 L_h \sigma_o^2}{2 \bar{\rho} n} } + {\textstyle \frac{\bar{\rho} n + 4 L_h \sigma_o^2}{2 \bar{\rho} n}}  \| \overline{h}( \Tprm_c^{(t)} ) \|^2 \} .
\end{split}}
\eeq
We observe the crude upper bound
\beq 
\begin{split}
& \gamma_{T+1} \| \grd V( \Tprm_c^{(T)} ) \| \leq d_0^{\frac{1}{2}} \, \gamma_{T+1}  \| \overline{h} ( \Tprm_c^{(T)} ) \| \\
& \leq \frac{ d_0^{\frac{1}{2}} }{2} \big( 1 + \gamma_{T+1}^2 \| \overline{h} ( \Tprm_c^{(T)} ) \|^2 \big) \leq \frac{ d_0^{\frac{1}{2}} }{2} + \frac{ d_0^{\frac{1}{2}} }{2} \sum_{t=0}^T \gamma_{t+1}^2 \| \overline{h} ( \Tprm_c^{(t)} ) \|^2.
\end{split}
\eeq
Define the constants
\beq \label{eq:cmk}
\begin{split}
{\rm C}_0^{\sf mk} & := K_{\PP} \Big\{ \frac{ d_0^{\frac{1}{2}} }{2} + \gamma_1 \| \grd V( \Tprm_c^{(0)} ) \| \Big\}, \\
{\rm C}_1^{\sf mk} & := K_{\PP}  L_V \frac{\bar{\rho} n (1 + 2 \sigma_h) + L_h (1 + 4\sigma_o^2)}{2 \bar{\rho} n} + d_0^{\frac{1}{2}} \big( \bar{L}_h \sigma_h^2 + \bar{L}_h \frac{4 \sigma_o^2}{ \bar{\rho} } + K_{\PP} \hat{a} \big), \\
{\rm C}_2^{\sf mk} & := d_0^{\frac{1}{2}} \bar{L}_h \big( 2 + \frac{L_h}{2n} + \frac{4 \sigma_o^2}{ \bar{\rho} } \big) + K_{\PP} \frac{d_0^\frac{1}{2}}{2} + K_{\PP} \Big( \hat{a} a^2 d_0^{\frac{1}{2}} + L_V \frac{\bar{\rho} n + 4 L_h \sigma_o^2}{2 \bar{\rho} n} \Big).
\end{split}
\eeq
Combining the terms yields
\beq \notag
\begin{split}
& \textstyle \EE \big[ \big| \sum_{t=0}^T \gamma_{t+1} \pscal{ \grd V( \Tprm_c^{(t)} ) }{ {\bm e}_0^{(t)} } \big| \big] \leq {\rm C}_0^{\sf mk} + {\rm C}_1^{\sf mk} \sum_{t=0}^T \gamma_{t+1}^2 + {\rm C}_2^{\sf mk} \sum_{t=0}^T \gamma_{t+1}^2 \| \overline{h} (\Tprm_c^{(t)} ) \|^2.
\end{split}
\eeq
This is the desired result for the lemma.

\section{Auxiliary Lemma}
\begin{Lemma} \label{lem:sum}
Let $\{ a_s \}_{s \geq 0}$ be an arbitrary sequence of non-negative number and $\rho \in (0,1)$ be a constant. For any $T \geq 0$, we have
\beq \label{eq:sum_lemma}
\sum_{t=0}^T \left( \sum_{s=0}^t a_s (1-\rho)^{t-s} \right)^2 
\leq \frac{2}{\rho} \sum_{t=0}^T a_t^2.
\eeq
\end{Lemma} 

\begin{proof}
We begin by expanding the summation on the l.h.s.~of \eqref{eq:sum_lemma} and observing the following upper bound:
\beq
\begin{split}
\sum_{t=0}^T \sum_{s=0}^t \sum_{q=0}^t a_s a_q (1-\rho)^{2t-q-s} & \leq \sum_{s=0}^T \sum_{q=0}^s a_s a_q (1-\rho)^{-q-s}  \sum_{t=s}^T (1-\rho)^{2t} \\
& \quad + \sum_{q=0}^T \sum_{s=0}^q a_s a_q (1-\rho)^{-q-s}  \sum_{t=q}^T (1-\rho)^{2t} 
\end{split}
\eeq
As $\sum_{t=s}^T (1-\rho)^{2t} \leq \frac{(1-\rho)^{2s}}{\rho}$, we have
\beq
\begin{split}
& \sum_{s=0}^T \sum_{q=0}^s a_s a_q (1-\rho)^{-q-s}  \sum_{t=s}^T (1-\rho)^{2t} \leq \frac{1}{2\rho} \sum_{s=0}^T \sum_{q=0}^s \big\{ a_s^2 + a_q^2 \big\} (1-\rho)^{s-q}
\end{split}
\eeq
Observe that 
\beq \notag
\sum_{s=0}^T a_s^2 \sum_{q=0}^s (1-\rho)^{s-q} = \sum_{s=0}^T a_s^2 \sum_{q=0}^s (1-\rho)^{q} \leq \sum_{s=0}^T \frac{a_s^2}{\rho}
\eeq
\beq \notag
\sum_{s=0}^T \sum_{q=0}^s  a_q^2 (1-\rho)^{s-q} = \sum_{q=0}^T a_q^2 \sum_{s=q}^T (1-\rho)^{s-q} \leq \sum_{s=0}^T \frac{a_s^2}{\rho}
\eeq
This shows
\beq \notag
\begin{split}
& \sum_{s=0}^T \sum_{q=0}^s a_s a_q (1-\rho)^{-q-s}  \sum_{t=s}^T (1-\rho)^{2t} \leq \frac{1}{\rho^2} \sum_{s=0}^T a_s^2
\end{split}
\eeq
By symmetry, we have $\sum_{q=0}^T \sum_{s=0}^q a_s a_q (1-\rho)^{-q-s}  \sum_{t=q}^T (1-\rho)^{2t} \leq \frac{1}{\rho^2} \sum_{s=0}^T a_s^2$. Adding this two bounds yields the desired result in \eqref{eq:sum_lemma}.
\end{proof}


{
\bibliographystyle{ieeetr}
\bibliography{cdc_dsa}
}

\end{document}